\documentclass[preprint,12pt]{elsarticle}
\graphicspath{{./Fig/}}

\usepackage{graphicx}
\usepackage{amssymb}
\usepackage[caption=false]{subfig}
\usepackage{float}
\usepackage{lineno}
\usepackage{amsmath}

\usepackage{indentfirst}
\usepackage{algorithm}
\usepackage{algcompatible}
\usepackage{mathtools}

\usepackage{hyperref}
\hypersetup{hidelinks}

\usepackage{color}
\usepackage{amssymb}
\usepackage{mathrsfs}
\usepackage{xcolor,soul}
\biboptions{sort&compress} 

\usepackage{ulem}

\usepackage{verbatim}

\newtheorem{theorem}{Theorem}
\newtheorem{lemma}{Lemma}
\newtheorem{proof}{Proof}
\newtheorem{definition}{Definition}
\newtheorem{remark}{Remark}
\newcommand{\de}{\partial}
\newcommand{\di}{{\rm d}}

\newcommand{\vect}[1]{\boldsymbol{#1}}
\newcommand{\tr}{\textcolor{red}} 

\journal{J COMPUT APPL MATH}

\begin{document}

\author[address]{Yang Wang \fnref{label1}}
\fntext[label1]{PhD Candidate, \href{mailto:yang.wang@polimi.it}{yang.wang@polimi.it}.}

\author[address]{Francesco Topputo \corref{cor} \fnref{label2}}
\fntext[label2]{Associate Professor, \href{mailto:francesco.topputo@polimi.it}{francesco.topputo@polimi.it}.}

\cortext[cor]{Corresponding author.}

\address[address]{Department of Aerospace Science and Technology, Politecnico di Milano,\\ Via La Masa 34, Milan, Italy, 20156. }

\begin{frontmatter}


\title{A Homotopy Method Based on Theory of Functional Connections}



%
\begin{abstract}
A method for solving zero-finding problems is developed by tracking homotopy paths, which define connecting channels between an auxiliary problem and the objective problem. Current algorithms' success highly relies on empirical knowledge, due to manually, inherently selected homotopy paths. This work introduces a homotopy method based on the Theory of Functional Connections (TFC). The TFC-based method implicitly defines infinite homotopy paths, from which the most promising ones are selected. A two-layer continuation algorithm is devised, where the first layer tracks the homotopy path by monotonously varying the continuation parameter, while the second layer recovers possible failures resorting to a TFC representation of the homotopy function. Compared to pseudo-arclength methods, the proposed TFC-based method retains the simplicity of direct continuation while allowing a flexible path switching. Numerical simulations illustrate the effectiveness of the presented method.
\end{abstract}

\end{frontmatter}



%
\section{Introduction}
\label{S:1}

The homotopy method is an effective technique used to tackle difficult zero-finding problems \cite{Easterling2018,Haberkorn2004,Bulirsch1991,hermant2011optimal,Ji2009}. By traversing a series of auxiliary problems, the homotopy method solves the objective problem by tracking the homotopy path, which is comprised of solutions of former \cite{allgower2003introduction}. There are two steps for designing an effective homotopy method. The first is to construct a homotopy function, while the second is to design an algorithm to track the implicitly defined homotopy path.

For what concerns the construction of the homotopy function, many variations can be found in literature. In \cite{Rahimian2011}, a combination of Newton function and fixed-point function is proposed. In \cite{Dai2003}, all isolated solutions of the cyclic-$n$ polynomial equations are found using polyhedral homotopy method. Newton homotopy method with adjustable auxiliary function is proposed in \cite{wu2006solving}. In \cite{PanLu-1822}, a double-homotopy method is used to construct discontinuous paths \cite{pan2018new}. Homotopy methods from control point of view are investigated in \cite{Ohtsuka1994,kotamraju2000stabilized}. All in all, the state of the art is to define homotopy functions that yield one or few homotopy paths. As the pre-defined homotopy path is not altered during the solution process, the success of the method relies on the empirical knowledge of the objective problem.

For what concerns the tracking strategies, there are two main categories: the piecewise-linear (PL) and the predictor-corrector (PC) continuation methods \cite{allgower2003introduction}. PL methods follow the path by building a piecewise linear approximation of the homotopy line. The search space is subdivided into cells, and the approximation is achieved by finding the solution at faces of cells \cite{haberkorn2004low}. PL methods pose less requirements on the underlying equations, but they are slower and less efficient for high-dimensional problems than PC methods \cite{allgower2003introduction}. The latter track the path through prediction and correction stages. The simplest and most commonly used PC method is the discrete continuation method (DCM) \cite{Haberkorn2004}. 

In DCM, the homotopy parameter varies monotonously at each step. The simplest predictor for DCM is to use the solution of the previous auxiliary problem. A variety of higher-order predictors, such as polynomial extrapolation \cite{allgower2003introduction} and Runge--Kutta methods \cite{bates2011efficient}, have been investigated. In \cite{Brown2016}, the monolithic homotopy method is formulated by integrating the predictor and corrector into a single component. An improvement of this method consists of using higher derivative information \cite{Brown2019}. Although DCM is straightforward and easy to implement, it fails when the homotopy path encounters unfavorable conditions, such as limit points (where the Jacobian matrix is ill-conditioned) or the path goes off to infinity \cite{PanLu-1822}. 

One enhanced PC method is the pseudo-arclength method (PAM) \cite{allgower2003introduction}. By reversing the homotopy path direction and augmenting the Jacobian matrix, PAM can effectively pass limit points \cite{allgower2003introduction}. Compared to DCM, PAM has a broader convergence domain, yet its implementation is more involved \cite{Yamamura1993}. However, PAM may still fail, e.g., when the homotopy path grows indefinitely \cite{wayburn1987homotopy}. This in turn calls for enhancements to improve the algorithmic robustness in homotopy methods.

The Theory of Connections (ToC) has been recently proposed to investigate arbitrary connections between points \cite{mortari2017theory,mortari2017least}. The Theory of Functional Connections (TFC) extends the ToC to the functional domain \cite{mortari2018theory,mortari2019multivariate,leake2019analytically,mai2019theory}. Inspired by the conceptual similarity between homotopy and connections, a TFC-based homotopy method is presented in this paper. TFC-based homotopy implicitly defines infinite homotopy paths that connect the auxiliary problems to the objective problem. This feature paves the way to enhance the algorithm performance by leveraging the freedom in the selection of the homotopy line. A two-layer method that combines DCM and TFC homotopy function is designed. Specifically, DCM is used in the first layer, while the second layer is triggered when continuation fails to advance on the current homotopy path. In the second layer, the TFC-based homotopy function is explored to search a different but feasible homotopy path. Thus, the devised method retains the easy implementation of DCM, while enabling flexible path switching. Several numerical examples are conducted to illustrate the effectiveness of the proposed method.

The paper is structured as follows. Section 2 introduces the fundamentals of homotopy methods. Section 3 outlines the TFC-based, DCM method. In Section 4, several numerical simulations are conducted. Conclusions are drawn in Section 5.



%
\section{Fundamentals of Homotopy Methods}

\subsection{Homotopy Function}
\label{S:2}

Consider the zero-finding problem
\begin{equation}
\label{eq:root_finding_Prob}
    \vect F(\vect x)=\vect 0
\end{equation}
where $\vect x \in \mathbb{R}^n $ and $\vect F : \mathbb{R}^n \to \mathbb{R}^n$ is a $\mathcal{C}^2$ function. Newton's method is widely used to solve problem \eqref{eq:root_finding_Prob}. However, it fails if the initial guess solution lies beyond its convergence domain, or singular points are encountered during iterations. These issues are likely in high-sensitive, nonlinear systems.

Homotopy is an effective strategy to solve difficult zero-finding problems, which lacks a priori knowledge on good initial guesses \cite{allgower2003introduction}. To solve Eq.\ \eqref{eq:root_finding_Prob}, one may define a homotopy or deformation function ${\vect \Gamma }({\kappa,\vect x}): \mathbb{R} \times \mathbb{R}^n \to \mathbb{R}^n $ such that 
\begin{equation}
    \label{eq:HM_conds}
    \vect\Gamma(0,\vect x) = \vect G (\vect x), \quad \quad \vect\Gamma(1,\vect x) = \vect F (\vect x)
\end{equation}
where $\kappa \in [0,1]$ is the homotopy parameter and $\vect G(\vect x): \mathbb{R}^n \to \mathbb{R}^n$ is a user-defined, auxiliary function. Typically, solving $\vect G(\vect x) = \vect 0$ is easier than solving Eq.\ \eqref{eq:root_finding_Prob}. The convex homotopy function is the commonly used form for $\vect\Gamma$:
\begin{equation}
\label{eq:traditional_func}
    \vect\Gamma(\kappa, \vect x) \coloneqq \kappa \ \vect F(\vect x) + (1 - \kappa ) \ \vect G(\vect x)
\end{equation}
Three types of homotopy are commonly used \cite{Rahimian2011}, depending on $\vect G$:
\begin{enumerate}
    \item Newton homotopy, $\vect G(\vect x) \coloneqq \vect F(\vect x)  - \vect F(\vect x_0)$
    \item Fixed-point homotopy, $\vect G(\vect x) \coloneqq \vect x  - \vect x_0$
    \item Affine homotopy, $\vect G(\vect x) \coloneqq A \left(\vect x  - \vect x_0\right)$
\end{enumerate}
where $\vect x_0$ is the solution to $\vect G(\vect x) = \vect 0$ and $A$ is a $n \times n$ matrix. 

Under regularity assumptions \cite{allgower2003introduction,Bhaya2013}, defining the homotopy function inherently generates a unique curve $\vect c(\theta) \coloneqq \vect\Gamma^{-1}(\vect 0): J \to \mathbb{R}^n$ for some open interval $J\subset\mathbb{R}$ starting from the initial solution $\vect x_0$, which contains points satisfying the consistency condition $\vect\Gamma(\kappa, \vect x) = \vect{0}$. The tracked solution curve in $\mathbb{R}^{n+1}$ is called homotopy path or zero curve. With reference to Fig.\ \ref{fig:homotopy_path_type}, the homotopy paths can be mainly classified in five Types \cite{watson2002probability}:

\begin{figure}
    \centering
    \scalebox{0.95}{\includegraphics{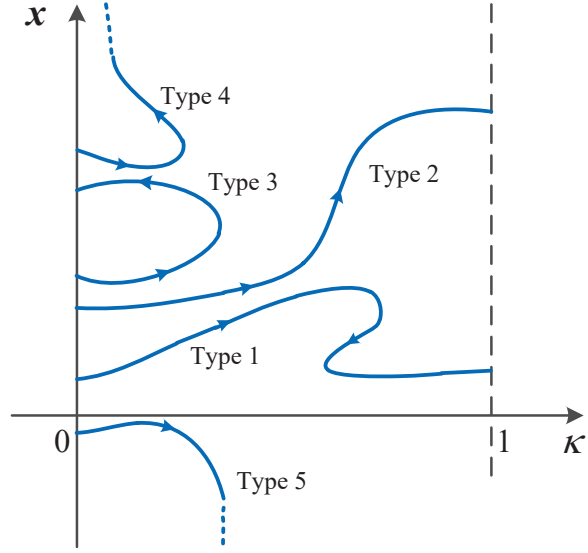}}
    \caption{Different types of homotopy paths $\vect x(\kappa)=\vect\Gamma^{-1}(\vect 0)$ starting from $\vect x(0)$.}
    \label{fig:homotopy_path_type}
\end{figure}

\begin{enumerate}
    \item[1)] The homotopy path ends in $\{1\} \times \mathbb{R}^n$, with non-monotonic $\kappa$;
    \item[2)] The homotopy path ends in $\{1\} \times \mathbb{R}^n$, with monotonic $\kappa$;
    \item[3)] The homotopy path returns to a solution of $\vect\Gamma(0,\vect x)$ in $\{0\} \times \mathbb{R}^n$;
    \item[4)] The homotopy path is unbounded, with non-monotonic $\kappa \in [0,1)$;
    \item[5)] The homotopy path is unbounded, with monotonic $\kappa \in [0,1)$.
\end{enumerate}

Homotopy methods attempt to track the homotopy path starting from $(0,\vect x^*(0))$ to $(1,\vect x^*(1))$. When this happens, one zero of Eq.\ \eqref{eq:root_finding_Prob} is found. The sufficient conditions for the existence of the homotopy path are given by probability-one homotopy theory \cite{watson2002probability,chow1978finding},  based on differential geometry concepts.

\begin{definition}[Transversality]
Let $U \subset \mathbb{R}^{n}$ and $V \subset \mathbb{R}^{p}$ be open sets, and let $\vect \rho$: $[0,1) \times U \times V \to \mathbb{R}^{n}$ be a $\mathcal{C}^2$ map. $\vect\rho$ is said to be transversal to zero if the Jacobian $D \vect \rho \in \mathbb{R}^{{n} \times (1+{n}+{p})} $ has full rank on $\vect\rho^{-1}(\vect 0)$.
\end{definition}

\begin{theorem}[Sard's theorem]
\label{theorem_1}
Let $\vect\rho$: $[0,1) \times U \times V \to \mathbb{R}^{n}$ be a $\mathcal{C}^2$ map. If $\vect\rho$ is transversal to zero, then for almost all $\vect a \in U$, the map
\begin{equation*}
    \vect\rho_{a}(\kappa,{\vect x}) \coloneqq \vect\rho( \kappa,\vect x, \vect a)
\end{equation*}
is also transversal to zero.
\end{theorem}

The parametrized Sard's theorem indicates that for almost all $\vect a \in U$, the zero set of $\vect\rho_a$ consists of smooth, nonintersecting curves~\cite{watson2002probability}. In the following, we take $U \equiv \mathbb{R}^n$ and $V \equiv \mathbb{R}^p$.

\begin{theorem}[Homotopy path]
\label{theorem_2}
Let $\vect\rho: [0,1) \times \mathbb{R}^n \times \mathbb{R}^p \to \mathbb{R}^n$ be a $\mathcal{C}^2$ map, and let $\vect\rho_a(\kappa,\vect x) = \vect\rho(\kappa,\vect x,\vect a)$. Suppose that:
\begin{enumerate}
\item[i)] for each fixed ${\vect a} \in \mathbb{R}^p$, ${\vect\rho}$ is transversal to zero;
\item[ii)] ${\vect\rho}_{a}(0,\vect x) = \vect 0$ has a unique nonsingular solution $\vect x(0)$;
\item[iii)] ${\vect \rho}_{a}(1,\vect x) = \vect F(\vect x)$;
\item[iv)] ${\vect \rho}^{-1}_{a}(\vect 0)$ is bounded;
\end{enumerate}
then, the solution curve reaches a point $(1,\vect x^*(1))$ such that $\vect F(\vect x^*(1)) = \vect 0$. Furthermore, if $D \vect F(\vect x^*(1))$ is invertible, then the homotopy path has finite arc length.
\end{theorem}

Transversality is hard to verify for arbitrary $\vect a \in \mathbb{R}^{p}$, and a proper $\vect a$ is required to construct the homotopy function. For example, fixed-point homotopy methods require selecting a proper $\vect x_0$. However, in current homotopy methods \cite{allgower2003introduction}, $\vect a$ is manually selected and it cannot vary during iterations. Thus, the success of the entire procedure relies heavily on the initial point chosen, and thus once again on the empirical knowledge of the problem.

\begin{remark}
The homotopy satisfying the hypotheses of Theorem\ \ref{theorem_2} is called a globally convergent probability-one homotopy \cite{watson2002probability}. Designing probability-one homotopy algorithms for general applications is still an open problem. Theorem\ \ref{theorem_2} is a guideline for robust homotopy algorithm design.
\end{remark}
\begin{remark}
	The $\mathcal{C}^2$ class is required for $\vect\rho$, and this condition cannot be relaxed \cite{watson2002probability}.
\end{remark}

\subsection{Path Tracking Methods}

Once the homotopy function is defined, the focus is on tracking its implicitly defined path. Two predictor-corrector methods are reviewed: discrete continuation method (DCM) and pseudo-arclength method (PAM).

\subsubsection{Discrete Continuation Method}

DCM tries to solve  $\vect\Gamma(\kappa, \vect x) = \vect 0$ with monotonous variation of $\kappa$ \cite{haberkorn2004low}. As shown in Fig.\ \ref{fig:DCM}, starting from initial solution at $\kappa = 0$, DCM solves the next solution on homotopy path using the former solution as initial guess. This process continues until the $\kappa = 1$ line is reached. DCM is simple and easy to implement, but it fails when the homotopy path exhibits limit points (Type 1, 3, 4) or goes off to infinity (Type 5). Limit points are points where the Jacobian $\vect\Gamma_{\vect x}({\kappa},\vect x)$ is singular, thus DCM cannot continue by monotonously varying $\kappa$\ \cite{Moore1991}. In Fig.\ \ref{fig:DCM}, the simple zero-order DCM method is shown. In principles, one can construct a higher-order predictor using polynomial extrapolation \cite{allgower2003introduction}. This could result in a more efficient algorithm, yet higher-order DCM will still fail at limit points. Another type of singular points are bifurcation points where homotopy path branches emanate \cite{Moore1991}. Bifurcation points are not considered in this work. 




\begin{figure}
    \centering
    \scalebox{0.9}{\includegraphics{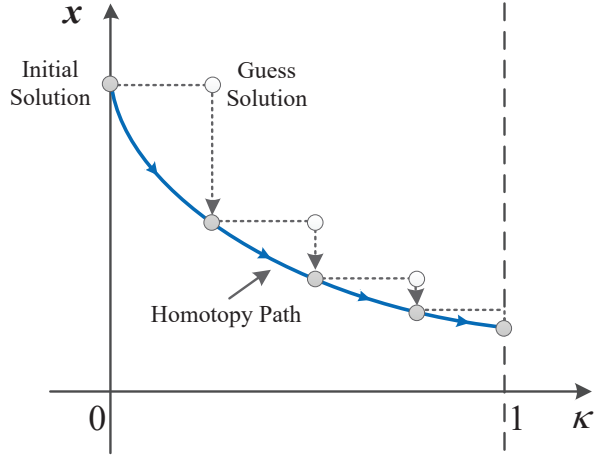}}
    \caption{Graphical interpretation of DCM.}
    \label{fig:DCM}
\end{figure}

\subsubsection{Pseudo-Arclength Method}

PAM is an alternative to pass limit points. Suppose that a solution point $(\kappa_i, \vect x_i)$ satisfies the consistency condition and its unit tangent direction $(\hat{\kappa}_i, \hat{\vect x}_i)$ is known, where the hat is the derivative w.r.t.\ the arclength $s$. In order to find the next solution point $(\kappa_{i+1}, \vect x_{i+1})$, the following augmented system is to be solved for $(\kappa, \vect x)$
\begin{equation}
\label{eqs:agu_eqs}
	\left\{
    \begin{array}{l}
    \vect\Gamma(\kappa, \vect x)  = \vect 0 \\[1mm]
    \left(\vect x - \vect x_i\right)^\top \hat{\vect x}_i + \left(\kappa - \kappa_i \right) \hat{\kappa}_i - \di s = 0
    \end{array}
    \right.
\end{equation}


The orientation of traversing is determined by the augmented Jacobian of system \eqref{eqs:agu_eqs} evaluated at $(\kappa_i, \vect x_i)$, that is,
\begin{equation*}
    \vect J_a (\kappa_i, \vect x_i) =
    \begin{bmatrix}
    \vect\Gamma_x (\kappa_i, \vect x_i) &     \vect\Gamma_\kappa (\kappa_i, \vect x_i) \\[1mm]
    \hat{\vect x}_i^\top & \hat{\kappa}_i
    \end{bmatrix}
\end{equation*}
 
\begin{figure}
    \centering
    \scalebox{0.9}{\includegraphics{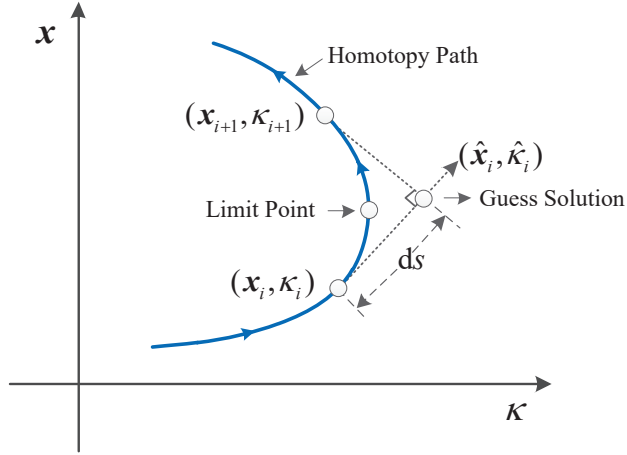}}
    \caption{Graphical interpretation of PAM near a limit point.}
    \label{fig:pseudo_arclength}
\end{figure}

The ability of PAM to pass a limit point is graphically shown in Fig.\ \ref{fig:pseudo_arclength}. When a limit point is approached, PAM attempts to track the homotopy path by predicting the solution along the tangent direction, and refining the solution until system \eqref{eqs:agu_eqs} is solved. Geometrically, the solution curve continues on the opposite $\kappa$ direction (in Fig.\ \ref{fig:pseudo_arclength}, $\kappa$ decreases across the limit point). PAM can elegantly satisfy condition $i)$ in Theorem \ref{theorem_2}, but it still fails when dealing with homotopy path Types 3--5. Compared to DCM, PAM has broader convergence domain, but its implementation is more involved \cite{Yamamura1993}.
%
\section{Theory of Functional Connection Homotopy Method}

\subsection{Theory of Functional Connections}

The Theory of Functional Connections (TFC) is the extension of the Theory of Connections (TOC) \cite{mortari2017theory}. The latter investigates the arbitrary connections between points by constructing a constrained function expressed in terms of an auxiliary function\ \cite{mortari2017theory,mortari2018theory}. It has the property that no matter what the auxiliary function is, the constrained function always satisfies a prescribed set of constraints. 

Suppose we define the scalar function
\begin{equation} \label{eq:TOCscalar2pts}
y(\eta) \coloneqq g(\eta) + \dfrac{\eta - \eta_0}{\eta_f - \eta_0} \left(y_f - g_f\right) + \dfrac{\eta_f - \eta}{\eta_f - \eta_0} \left(y_0 - g_0\right)
\end{equation}
where $y(\eta)$ and $g(\eta)$ are the \textit{constrained function} and \textit{auxiliary function}, respectively, whereas $\eta \in \left[\eta_0, \eta_f \right]$ is the independent variable. It is easy to verify that Eq.\ \eqref{eq:TOCscalar2pts} inherently satisfies $y(\eta_0) = y_0$ and $y(\eta_f) = y_f$ regardless of the specific choice of $g(\eta)$ (note that $g_0=g(\eta_0)$ and $g_f=g(\eta_f)$). Therefore, the line $y(\eta)$ will always connect the points $P_0=(\eta_0, y_0)$ and $P_f=(\eta_f, y_f)$. Eq.\ \eqref{eq:TOCscalar2pts} is the generalization of interpolation formulae: it is not the interpolating expression for a class of functions but for all functions \cite{mortari2017theory}. 


In the multi-dimensional case, the two-point condition is
\begin{equation}
\label{eq:bcs_cons}
        \vect y(\eta_0) = \vect y_0, \qquad
        \vect y(\eta_f) = \vect y_f
\end{equation}
where $\vect y \in \mathbb{R}^n$. The general expression of the constrained function $\vect y(\eta)$ is
\begin{equation}
    \label{eq:expr_y}
    \vect y(\eta) = \vect g(\eta) + P_1(\eta) \vect c_1 + P_2(\eta) \vect c_2
\end{equation}
where $P_{1,2} : \mathbb{R} \to \mathbb{R}^{n \times n}$ are matrices whose elements are scalar-valued functions of $\eta$, while $\vect c_{1,2} \in \mathbb{R}^n$ are constant vectors of weights \cite{mortari2017theory}. Substituting Eq.\ \eqref{eq:bcs_cons} into Eq.\ \eqref{eq:expr_y} and solving for $\vect c_{1,2}$ yields
\begin{equation}
\label{eq:expr_c}
    \begin{pmatrix}
    \vect c_1\\
    \vect c_2
    \end{pmatrix}
    =
    \begin{bmatrix}
    P_1(\eta_0) & P_2(\eta_0) \\
    P_1(\eta_f) & P_2(\eta_f)
    \end{bmatrix}^{-1}
    \begin{pmatrix}
    \vect y_0 - \vect g_0\\
    \vect y_f - \vect g_f
    \end{pmatrix}
    =
    \begin{bmatrix}
    Q_{11} & Q_{12}\\
    Q_{21} & Q_{22}
    \end{bmatrix}
    \begin{pmatrix}
    \vect y_0 - \vect g_0\\
    \vect y_f - \vect g_f
    \end{pmatrix}
\end{equation}
where again $\vect g_0 = \vect g(\eta_0)$ and $\vect g_f = \vect g(\eta_f)$. Moreover
\begin{equation}
\label{eq:expr_q}
    \begin{array}{rcl}
    Q_{11} &=& \left[P_{1}(\eta_0) - P_{2}(\eta_0) P_{2}^{-1}(\eta_f) P_{1}(\eta_f)\right]^{-1}\\[1mm]
    Q_{21} &=& - P_{2}^{-1}(\eta_f) P_{1}(\eta_f) \ Q_{11}\\[1mm]
    Q_{12} &=& - P_{1}^{-1}(\eta_0) P_{2}(\eta_0) \ Q_{22}\\[1mm]
    Q_{22} &=& \left[P_{2}(\eta_f) - P_{1}(\eta_f) P_{1}^{-1}(\eta_0) P_{2}(\eta_0)\right]^{-1}
    \end{array}
\end{equation}
The selection of $P_{1,2}(\eta)$ in Eq.\ \eqref{eq:expr_y} must ensure the existence of $Q_{ij}$ in Eq.\ \eqref{eq:expr_q}. Substituting Eqs.\ \eqref{eq:expr_c} and \eqref{eq:expr_q} into Eq.\ \eqref{eq:expr_y} gives the general form of constrained function
\begin{equation}
    \label{eq:constrained_func_y}
    \vect y(\eta) = \vect g(\eta) + \sum_{i=1}^{2} P_i(\eta) Q_{i1} (\vect y_0 - \vect g_0) + \sum_{i=1}^{2} P_i(\eta) Q_{i2} (\vect y_f - \vect g_f)
\end{equation}

The constrained function $\vect y(\eta)$ in Eq.\ \eqref{eq:constrained_func_y} defines arbitrary connection paths between $\vect y_0$ and $\vect y_f$ produced by the infinitely possible choices of $\vect g(\eta)$. The constrained function for arbitrary boundary conditions can also be established~\cite{mortari2017theory}. The Theory of Functional Connections (TFC) extends the idea above to construct the constrained function on a functional domain \cite{mortari2018theory}. 


\subsection{TFC-Based Homotopy Function}

From a geometrical point of view, the homotopy function defines the solution curve connecting the two zero-finding problems defined at the boundaries of $\kappa$, which satisfy Eq.\ \eqref{eq:HM_conds}. Analogously, the constrained function in the TFC connects points at the boundaries of $\eta$. Interpreting the constrained function as describing an homotopy path is therefore natural.

In Eq.\ \eqref{eq:constrained_func_y}, replacing the constrained function $\vect y(\eta)$ by the homotopy function $\vect \Gamma(\eta,\vect x)$, and $\vect y_0$, $\vect y_f$ by $\vect G(\vect x)$, $\vect F(\vect x)$, respectively, we have
\begin{equation}
    \label{eq:constrained_func_gamma}
    \vect\Gamma(\eta,\vect x) = \vect g(\eta) + \sum_{i=1}^{2} P_i(\eta) Q_{i1} (\vect G(\vect x)-\vect g_0) + \sum_{i=1}^{2} P_i(\eta) Q_{i2} (\vect F(\vect x)-\vect g_f)
\end{equation}
The auxiliary function $\vect g(\eta)$ can be expressed as a linear combination of basis functions with corresponding weights, that is
\begin{equation}
    \label{eq:expr_auxiliary_func}
    \vect g(\eta) = \Omega \vect h(\eta)
\end{equation}
where $\vect h(\eta): \mathbb{R} \to \mathbb{R}^{m}$ is the vector of basis functions, whereas $\Omega \in \mathbb{R}^{n \times m}$ is the matrix of weights. Note that $\vect g_0 = \Omega \vect h_0$ and $\vect g_f = \Omega \vect h_f$, where $\vect h_0 = \vect h(\eta_0)$ and $\vect h_f = \vect h(\eta_f)$. A linear map between $\kappa \in \left[0,1\right]$ and $\eta \in [\eta_0, \eta_f]$ is also used:
\begin{equation}
\label{eq:ke}
\eta(\kappa) = (1-\kappa) \, \eta_0  + \kappa \, \eta_f
\end{equation}

Substituting Eqs.\ \eqref{eq:expr_auxiliary_func} and \eqref{eq:ke} into Eq.\ \eqref{eq:constrained_func_gamma} yields
\begin{equation}
    \label{eq:constrained_func_f_2}
        \vect\Gamma(\kappa,\vect x,\Omega) = \Omega \vect h(\kappa) + \sum_{i=1}^{2} P_i(\kappa) Q_{i1} \left(\vect G(\vect x)-\Omega \vect h_0\right) + \sum_{i=1}^{2} P_i(\kappa) Q_{i2} \left(\vect F(\vect x)-\Omega \vect h_f\right)
\end{equation}
Notice that $\vect\Gamma$ in Eq.\ \eqref{eq:constrained_func_f_2}, beside the natural dependence on $\kappa$ and $\vect x$, is also a function of the free parameter $\Omega$, which can be varied to steer the solution curve from $\vect G^{-1}(\vect 0)$ to $\vect F^{-1}(\vect 0)$. 

It is convenient to isolate in Eq.\ \eqref{eq:constrained_func_f_2} the part depending on $\kappa$ and $\vect x$ only
\begin{equation}
    \label{eq:constrained_func_f_3}
    \vect\Gamma(\kappa, \vect x, \Omega) =\Omega  \left(\vect h(\kappa) - \sum_{i=1}^{2} P_i(\kappa) Q_{i1} \vect h_{0} - \sum_{i=1}^{2} P_i(\kappa) Q_{i2} \vect h_{f}\right) + \vect \Gamma_0(\kappa,\vect x)
\end{equation}
where
\begin{equation} \label{eq:Gamma_0}
\vect \Gamma_0(\kappa,\vect x) \coloneqq \sum_{i=1}^{2} P_i(\kappa) Q_{i1} \vect G(\vect x) + \sum_{i=1}^{2} P_i(\kappa) Q_{i2} \vect F(\vect x)
\end{equation}

By taking the partial derivative of Eq.\ \eqref{eq:constrained_func_f_3} w.r.t.\ $\vect x$, we find that
\begin{equation*}
	\dfrac{\de \vect\Gamma(\kappa,\vect x,\Omega)}{\de \vect x} = \dfrac{\de \vect \Gamma_0(\kappa,\vect x)}{\de \vect x}
\end{equation*}
indicating that when a limit point is encountered both Jacobian matrices are singular, regardless of the selection of $\Omega$. In order to regularize $\de\vect\Gamma/\de\vect x$ by varying $\Omega$, we let the basis functions $\vect h$ to depend on the present solution $\vect x$ as well; that is, $\vect h = \vect h(\kappa, \vect x)$. Thus, Eq.\ \eqref{eq:constrained_func_f_3} becomes
\begin{equation}
    \label{eq:constrained_func_f_final}
\vect\Gamma(\kappa,\vect x,\Omega) = \Omega \vect\Gamma_{\Omega}(\kappa,\vect x) + \vect\Gamma_0(\kappa,\vect x)
\end{equation}
where
\begin{equation}
    \label{eq:constrained_func_f_final_aux}
    \vect\Gamma_{\Omega}(\kappa,\vect x) \coloneqq \vect h(\kappa,\vect x) - \sum_{i=1}^{2} P_i(\kappa) Q_{i1} \vect h_{0}(\vect x) - \sum_{i=1}^{2} P_i(\kappa) Q_{i2} \vect h_{f}(\vect x)
\end{equation}
Inspired by Eq.\ \eqref{eq:constrained_func_f_final}, the formal definition of TFC-based homotopy is given.

\begin{definition}[TFC-based homotopy function]
\label{def:tfc}
Let $\hat{\vect \rho}(\kappa,\vect x,\vect\varepsilon,\vect a)$ : $ [0,1) \times \mathbb{R}^n \times \mathbb{R}^q \times \mathbb{R}^p \to \mathbb{R}^n$ be a $\mathcal{C}^2$ map, and let $\hat{\vect \rho}_a(\kappa,\vect x,\vect\varepsilon) = \hat{\vect \rho}(\kappa,\vect x,\vect\varepsilon,\vect a)$ for fixed $\vect a$. $\hat{\vect \rho}_a(\kappa,\vect x,\vect\varepsilon)$ is called TFC-based homotopy function if 
\begin{enumerate}
	\item[i)] it automatically satisfies the boundary conditions 
	\begin{equation*}
	\hat{\vect \rho}_a(0,\vect x,\vect\varepsilon) = \vect G (\vect x) \qquad \textrm{and} \qquad \hat{\vect \rho}_a(1,\vect x,\vect\varepsilon) = \vect F (\vect x)
	\end{equation*}
	for arbitrary $\vect\varepsilon$;
	\item[ii)] $\forall\kappa \in (0,1)$ and $\forall\vect x \in \mathbb{R}^n$, $\exists\ \vect \epsilon$ such that $\de \hat{\vect \rho}_a(\kappa,\vect x,\vect\varepsilon)/\de \vect x$ is regular.
\end{enumerate}
\end{definition}

In traditional homotopy methods (e.g., Newton homotopy), the term $\vect a$ in the homotopy function $\vect \rho_a (\kappa, \vect x)$ in Theorem \ref{theorem_1} is set at the beginning of the continuation procedure (e.g., by providing the solution $\vect x(0)$ to the initial problem $\vect G(\vect x) = \vect 0$) and so is the homotopy path. The TFC-based homotopy function $\hat{\vect \rho}_a(\kappa,\vect x,\vect\varepsilon)$ is the generalization of $\vect \rho_a (\kappa, \vect x)$. Here, although $\vect a$ is fixed, $\vect \epsilon$ brings in flexibility in the homotopy path while not affecting the boundary conditions, Eq.\ \eqref{eq:HM_conds}. The TFC-based homotopy function implicitly defines infinite homotopy paths because of the infinite possible selections of $\vect \epsilon$. Moreover, condition ii) in Definition \ref{def:tfc} enables regularizing the path by varying $\vect \epsilon$. Therefore, it is a tool to recover improperly defined paths, by detecting them and switching to different, yet feasible, homotopy paths.

Equation \eqref{eq:constrained_func_f_final} provides a general form of TFC-based homotopy function. Here, $\vect \Gamma_0(\kappa,\vect x)$ is equivalent to $\vect \rho_a (\kappa, \vect x)$ and $\Omega$ can be seen as $\vect \epsilon$ (see Appendix). Let $\tau = e^{\eta_0 - \eta_f}$, the following three examples are given based on different choice of $P_{1,2}(\eta)$
\begin{enumerate}
	\item For $P_1 = I$ and $P_2 = {\eta} I$
	\begin{equation}
	\label{eq:ToCHM_case_1}
	\vect\Gamma (\kappa, \vect x,\Omega) = \Omega\left(\vect h(\kappa,\vect x) +(\kappa  - 1)\vect h_{0}(\vect x) -\kappa \vect h_{f}(\vect x) \right) + \vect\Gamma_0 (\kappa, \vect x)
	\end{equation}
	\item For $P_1 = I$ and $P_2 = e^\eta I$
	\begin{equation}
	\begin{aligned}
	\label{eq:ToCHM_case_2}
	\vect\Gamma (\kappa, \vect x,\Omega) = \Omega \left(\vect h(\kappa, \vect x) - \dfrac{1-\tau^{(1-\kappa)}}{1-\tau} \vect h_{0}(\vect x) -\dfrac{-\tau + \tau^{(1-\kappa)}}{1-\tau} \vect h_{f}(\vect x)  \right) + \vect \Gamma_0 (\kappa , \vect x)
	\end{aligned}
	\end{equation}
	\item For $P_1 = I$ and $P_2 = e^{-\eta} I$
	\begin{equation}
	\label{eq:ToCHM_case_3}
	\vect\Gamma (\kappa, \vect x,\Omega) = \Omega\left( \vect h(\kappa,\vect x) - \dfrac{\tau-\tau^\kappa}{\tau-1} {\vect h_{0}}(\vect x) -\dfrac{-1 + \tau^\kappa}{\tau-1}  \vect h_{f}(\vect x) \right) + \vect \Gamma_0 ( \kappa, \vect x)
	\end{equation}
\end{enumerate}


\subsection{Regularization}

This section shows the sufficient conditions for point ii) in Definition \ref{def:tfc}. 
\begin{lemma}
	\label{lemma_1}
	Suppose that a matrix $A \in \mathbb{R}^{n \times n}$ is the product of two matrices $B \in \mathbb{R}^{n \times m}$ and $C \in \mathbb{R}^{m \times n}$; $A = BC$. If $m < n$, then $A$ is singular.
\end{lemma}
\begin{proof}
Consider the linear equation
\begin{equation*}
C\vect x = \vect 0
\end{equation*}
if $m < n$, the number of equations is less than that of unknowns, thus there exists nonzero solution $\Tilde{\vect x}$ such that 
\begin{equation*}
C \Tilde{\vect x} = \vect 0
\end{equation*}
then
\begin{equation*}
BC \Tilde{\vect x} = A \Tilde{\vect x} = \vect 0
\end{equation*}
indicating that the matrix $A$ is singular. \qed
\end{proof}
\begin{lemma}
\label{lemma_2}
If $A \in \mathbb{R}^{m \times n}$ is full row rank and $m \leq n$, then $B=AA^\top \in \mathbb{R}^{m \times m}$ is regular. 
\end{lemma}
\begin{proof}
	Consider the linear equation
	\begin{equation*}
		B \vect x = AA^\top \vect x = \vect 0
	\end{equation*}
	which equals to
	\begin{equation*}
		\vect x^\top AA^\top \vect x = \left(A^\top \vect x\right)^\top A^\top \vect x = \vect 0\ \rightarrow \ A^\top \vect x = \vect 0
	\end{equation*}
Since $m \leq n$ and $A$ is full row rank, thus $\vect x = \vect 0$. Therefore, $B$ is regular. \qed
\end{proof}

\begin{theorem}
	\label{theorem_3}
	Let $\vect\Gamma(\kappa,\vect x,\Omega) = \Omega \vect\Gamma_{\Omega}(\kappa,\vect x) + \vect\Gamma_0(\kappa,\vect x)$ be the TFC-based homotopy function, and let $m=n$ . If ${\de \vect\Gamma_{\Omega} (\kappa, \vect x)}/{\de \vect x} \in \mathbb{R}^{m \times n}$ is regular, then $\exists\ \Omega \in \mathbb{R}^{n \times m}$ such that ${\de \vect\Gamma (\kappa , \vect x,\Omega)}/{\de \vect x}$ is regular.
\end{theorem}
\begin{proof}
	Taking the derivative of Eq.\ \eqref{eq:constrained_func_f_final} w.r.t.\ $\vect x$ yields
	\begin{equation*}
	\label{eq:derivative_ToC}
	\dfrac{\de \vect\Gamma (\kappa , \vect x,\Omega)}{\de \vect x} 
	=\Omega \dfrac{\de \vect\Gamma_{\Omega} (\kappa, \vect x)}{\de \vect x} + \dfrac{\de \vect\Gamma_0 (\kappa, \vect x)}{\de \vect x}
	\end{equation*}
	Applying singular value decomposition to ${\de \vect\Gamma_{0} (\kappa, \vect x)}/{\de \vect x}$, there exists
	\begin{equation*}
	\dfrac{\de \vect\Gamma_{0} (\kappa, \vect x)}{\de \vect x} = U^\top \left[
	\begin{matrix}
    \Sigma_1 & \\
	 & \Sigma_2
	\end{matrix}\right] V
	\end{equation*}
	where $\Sigma_1$ are nonzero singular values, and $\Sigma_2$ are zero singular values if ${\de \vect\Gamma_{0} (\kappa, \vect x)}/{\de \vect x}$ is singular. $U$ and $V$ are corresponding singular vectors. We can construct a regular matrix $S \in \mathbb{R}^{n \times n}$ as
	\begin{equation*}
	S = U^\top \left[
	\begin{matrix}
	\Lambda_1 &  \\
	 & \Lambda_2
	\end{matrix}\right] V
	\end{equation*}
	where $\Lambda_1$ and $\Lambda_2$ are non-zero singular values. There always exists $ \Lambda_1$ and $\Lambda_2$ such that the matrix
	\begin{equation*}
	\dfrac{\de \vect\Gamma (\kappa , \vect x,\Omega)}{\de \vect x} = S + \dfrac{\de \vect\Gamma_{0} (\kappa, \vect x)}{\de \vect x} = U^\top \left[
	\begin{matrix}
	\Lambda_1 + \Sigma_1 &  \\
	 & \Lambda_2 + \Sigma_2
	\end{matrix}\right] V \in \mathbb{R}^{n \times n}
	\end{equation*}
	is regular. Let $S \coloneqq \Omega\ {\de \vect\Gamma_{\Omega} (\kappa, \vect x)}/{\de \vect x}$. From Lemma \ref{lemma_1}, this requires $m \geq n$. Since ${\de \vect\Gamma_{\Omega} (\kappa, \vect x)}/{\de \vect x}$ is full rank and $m = n$, from Lemma \ref{lemma_2}, ${\exists}\ \Omega$ such that
	\begin{equation*}
	\Omega = S \left(\dfrac{\de \vect\Gamma_{\Omega} (\kappa, \vect x)}{\de \vect x}\right)^\top \left[ \left(\dfrac{\de \vect\Gamma_{\Omega} (\kappa, \vect x)}{\de \vect x}\right) \left(\dfrac{\de \vect\Gamma_{\Omega} (\kappa, \vect x)}{\de \vect x}\right)^\top \right]^{-1} 
	\end{equation*}
	\qed
\end{proof}

According to Theorem~\ref{theorem_3}, the following criteria are provided. Firstly, $m = n$. Secondly, the selection of $\vect h(\kappa,\vect x)$ should avoid zero elements for any possible values of $\vect x$. Non-zero functions such as exponential functions are preferred to construct each element of $\vect h(\kappa,\vect x)$. Thirdly, the selection of $\vect h(\kappa,\vect x)$ should consider the concrete form of TFC homotopy function. In Eqs.\ \eqref{eq:ToCHM_case_1}--\eqref{eq:ToCHM_case_3}, $\vect h(\kappa,\vect x)$ should be nonlinear in $\kappa$ to ensure the explicit dependence of $\vect\Gamma_{\Omega}(\kappa,\vect x)$ on $\kappa$.

\subsection{A Two-Layer TFC-based DCM Method}

Following the definition of the TFC-based homotopy function in Eq.\ \eqref{eq:constrained_func_f_final}, a two-layer DCM method is proposed.
\begin{figure}
	\centering
	\scalebox{1}{\includegraphics{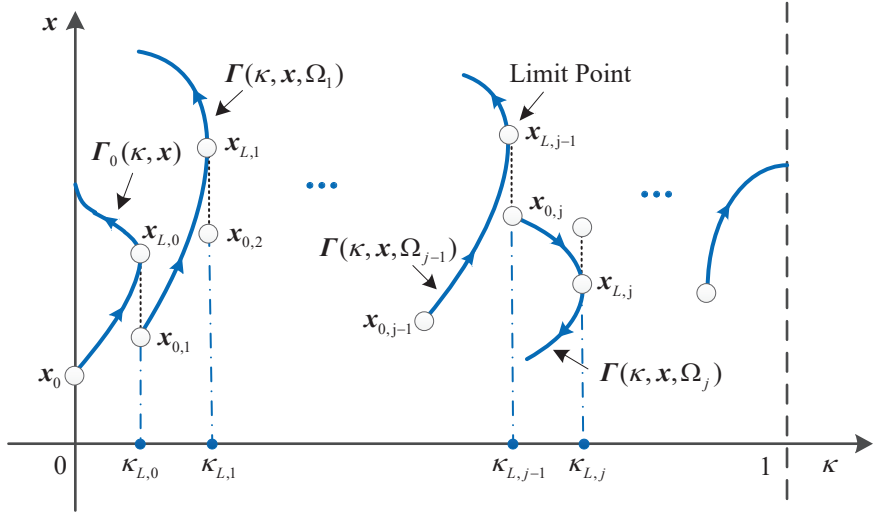}}
	\caption{Graphical layout of the singular point management.}
	\label{fig:TFCHomotopyMethodShow}
\end{figure}

\subsubsection{Singular Point Management}

Fig.~\ref{fig:TFCHomotopyMethodShow} illustrates the method, with a focus on limit point management. Starting from $\vect x_0$ at $\kappa = 0$, the DCM is used first to track the initial homotopy path, defined by $\vect \Gamma_0(\kappa,\vect x)$. When a limit point $\vect x_{L,0}$ is encountered at $\kappa_{L,0}$, another feasible homotopy path defined by $\vect \Gamma(\kappa, \vect x, \Omega_1)$ is found by searching for a proper $\Omega_1$. Then, the new starting point $\vect x_{0,1}$ at $\kappa_{L,0}$ triggers a new homotopy path, again tracked by DCM. At $\vect x_{L,1}$, the new homotopy path defined by $\vect \Gamma(\vect \kappa, \vect x, \Omega_2)$ is found and tracked. This process is repeated until the line $\kappa = 1$ is reached.

In general, suppose that the DCM encounters a limit point $\vect x_{L,j-1}$ at $\kappa_{L,j-1}$ while tracking the homotopy path defined by $\vect \Gamma(\kappa, \vect x, \Omega_{j-1})$. The goal is to switch to a new solution curve by finding a new homotopy path defined by $\vect \Gamma(\kappa,\vect x,\Omega_j)$ starting from $\vect x_{0,j}$ at $\kappa_{L,j-1}$. The unknown variables for the $j$-th homotopy path are $\Omega_j$ and $\vect x_{0,j}$; that is, a total of $(m+1)\times n$ unknowns against the $n$-dimensional consistency condition. The problem is clearly underdetermined, and therefore $\Omega_j$ and $\vect x_{0,j}$ are found by solving an optimization problem.

The main feature sought in a candidate homotopy path are feasibility and an easy progression of the DCM. Ideally, one may want to switch to a new feasible horizontal path, which would easily lead to the solution of the objective problem ($\kappa=1$). In this respect, the projected $\|\vect\Gamma\|_2$ error trend along a candidate homotopy path is considered. In Fig.~\ref{fig:PerformanceIndexp}, the projected error is discerned into a near-side error, $\vect \Gamma( \kappa_{L,j-1} + \Delta \kappa,\vect x,\Omega_j)$, and a far-side error $\vect \Gamma(\min(\kappa_{L,j-1} + i \zeta \Delta \kappa,1),\vect x,\Omega_j)$. The former is minimized to ease restart of the DCM, while the latter is weighted to select a mild path. The problem is therefore to
\begin{equation} \label{eq:opt_problem}
	\min_{\Omega_j,\vect x_{0,j}} J \quad \textrm{s.t.} \quad \vect c_{\rm eq} = \vect 0
\end{equation}
where
\begin{equation} \label{eq:objfcn}
J \coloneqq \|\vect\Gamma\left(\min(\kappa_{L,j-1} + \Delta \kappa, 1),\vect x, \Omega_j \right)\|_2 + \sum_{i=1}^{N}\gamma^i \|\vect\Gamma\left( \min(\kappa_{L,j-1} + i \zeta \Delta \kappa,1),\vect x,\Omega_j
    \right)\|_2
\end{equation}
and
\begin{equation} \label{eq:nlcon}
	\vect c_{\rm eq} \coloneqq
	\begin{cases} 
	\vect 1_{n \times 1}, & \mbox{if} ~ \left|\det\left({\de \vect \Gamma\left(\kappa_{L,j-1},\vect x_{0,j},\Omega_j \right)}/{\de \vect x}\right)\right| \leq \delta \\
	\vect\Gamma \left(\kappa_{L,j-1},\vect x_{0,j},\Omega_j \right), & \mbox{otherwise}
	\end{cases}
\end{equation}
In Eq.\ \eqref{eq:objfcn}, $\gamma \in [0,1)$ is a discount factor, $\zeta$ is the predicted horizon, and $N$ is the number of predicted points. An artificial violation of the equality constraint in Eq.\ \eqref{eq:nlcon} is introduced to avoid near-singular paths. Moreover, $\Omega_{j-1}$ and $\vect x_{L,j-1}$ are taken as initial guess for the optimization problem in Eq.\ \eqref{eq:opt_problem}.

\begin{figure}
	\centering
	\scalebox{1}{\includegraphics{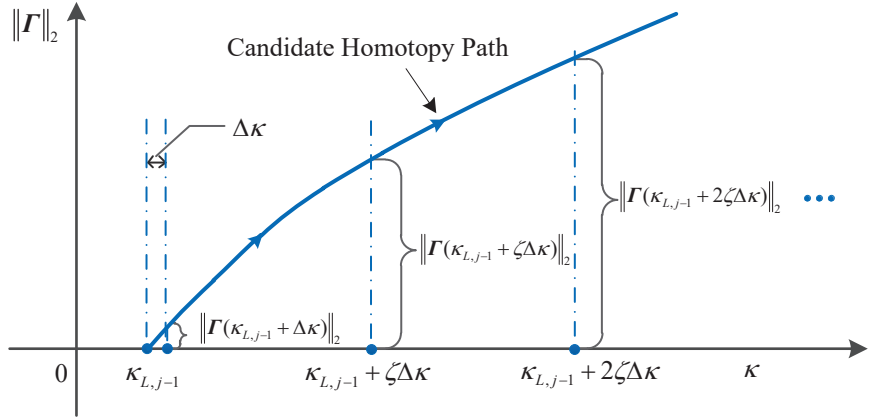}}
	\caption{Error trend along  a candidate homotopy path.}
	\label{fig:PerformanceIndexp}
\end{figure}

\subsubsection{Indefinite Growth Management}

Beside tackling limit points, paths of Type 5 in Fig.\ \ref{fig:homotopy_path_type} are also considered. As shown in Fig.\ \ref{fig:ToC_homotopy_comp_infinite_case}, indefinite growth is managed through thresholding. An a-priori threshold $T_h$ on $\| \vect x \|_{\infty}$ is set. Once the homotopy path crosses the threshold line, the second layer is triggered to switch to an alternative, feasible homotopy path.

In Fig.\ \ref{fig:ToC_homotopy_comp_infinite_case}, when the initial homotopy path exceeds $T_h$, the solution point $\vect x_{{\rm I},0}$ at $\kappa_{{\rm I},0}$ is detected. This is used as initial guess to solve the optimization problem in Eq.\ \eqref{eq:opt_problem}, and a new homotopy path (using $\Omega_1$ and starting from $\vect x_{0,1}$) is tracked. If this new homotopy path exceeds $T_h$ (Failed Case 1) or the solver fails to converge (Failed Case 2), the solution point near but below $T_h/2$ is considered, until a new feasible path is found. Failed Case 2 may happen because the homotopy path tends to infinity and thus ${\de \vect \Gamma\left(\kappa,\vect x,\Omega \right)}/{\de \vect x}$ tends to be singular. In Fig.\ \ref{fig:ToC_homotopy_comp_infinite_case}, the new homotopy path defined by $\vect \Gamma(\kappa,\vect x,\Omega_3)$ starting from $\vect x_{0,3}$ at $\kappa_{{\rm I},2}$ is found by using $T_h/4$. 

The algorithmic rationale of the two-layer, TFC-based homotopy method is summarized in Algorithm \ref{algo:TFC_Homotopy_Algorithm}.

\begin{figure}
    \centering
    \scalebox{0.9}{\includegraphics{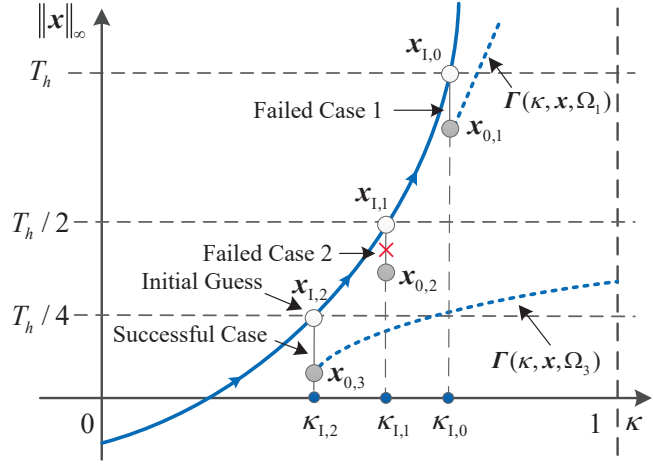}}
    \caption{Graphical layout of the indefinite growth management.}
    \label{fig:ToC_homotopy_comp_infinite_case}
\end{figure}



%

\begin{algorithm}[htp]
	\caption{Two-layer TFC-based DCM Algorithm}
	\label{algo:TFC_Homotopy_Algorithm}
	\begin{algorithmic}[1]
		\REQUIRE{$\Delta \kappa_d$ (default $\kappa$ step), $\vect h(\kappa,\vect x)$, $\vect G(\vect x)$, and $T_h$. }
		\ENSURE{Solution to $\vect F(\vect x) = \vect 0$.}
		\STATE Set $\kappa=0$, $\kappa_{\rm old} = 0$, $j=0$, $\Delta \kappa = \Delta \kappa_d $, and $\Omega_0 = 0_{n \times n}$.
		\STATE Solve the auxiliary problem $\vect G(\vect x) = \vect 0$. 
		\WHILE{$\kappa < 1$}
				\STATE $\kappa := \kappa + \Delta \kappa$.
				\STATE Solve the zero-finding problem $\vect \Gamma(\kappa,\vect x,\Omega_{j}) = \vect 0$.
				\IF{Converged but crossed $T_h$ line.}
					\STATE Solve the optimization problem Eq.~\eqref{eq:opt_problem}.
					\STATE Switch to the new homotopy path $\vect \Gamma(\kappa,\vect x, \Omega_{j+1})$,	$j:=j+1$.
					\STATE $\Delta \kappa := \min(1-\kappa,\Delta \kappa_d)$, $\kappa_{\rm old} := \kappa$.
				\ELSE
				
				\IF{Converged}
					\STATE $\Delta \kappa := \min(1-\kappa,\Delta \kappa_d)$. $\kappa_{\rm old} := \kappa$.
				\ELSE
					\IF{The zero-finding problem fails for less than 3 times}
						\STATE $\Delta \kappa := \Delta \kappa/2$.  $\kappa := \kappa_{\rm old}$. 
					\ELSE
						\STATE Solve the optimization problem Eq.~\eqref{eq:opt_problem}.
						\STATE Switch to the new homotopy path $\vect \Gamma(\kappa,\vect x,\Omega_{j+1})$, $j:=j+1$.
						\STATE $\Delta \kappa := \min(1-\kappa,\Delta \kappa_d)$, $\kappa_{\rm old} := \kappa$.
					\ENDIF
				\ENDIF
				\ENDIF
		\ENDWHILE
		
	\end{algorithmic}
\end{algorithm}

%
\section{Numerical Demonstration}

In this section, three numerical experiments with increasing difficulty are performed using the TFC-based DCM method. To ease assessment of the developed algorithm, the outcome of each problem is compared to the solution obtained using PAM. The homotopy function in Eq.\ \eqref{eq:ToCHM_case_1} is used in all problems. The zero-finding and optimization problems are solved using Matlab's \verb|fzero| and \verb|fmincon| implementing interior-point method, respectively. In both algorithms, the function residual (\verb|TolFun|) and solution tolerance (\verb|TolX|) are both set to $10^{-12}$. All test cases have been performed using Matlab R2019a with Intel Core i7-9750H CPU @2.60 GHz, Windows 10 operating system. The parameters of the optimization problem in Eqs.\ \eqref{eq:objfcn}--\eqref{eq:nlcon} are $\gamma = 0.5$, $\zeta = 15$, $N = 2$, and $\delta = 1 \times 10^{-4}$. A limit point is supposed to be encountered when the zero-finding problem fails for $3$ consecutive times, and half of the $\Delta \kappa$ step is taken.

\subsection{Algebraic Zero-Finding Problem} \label{sec:example1}

The zero of the following two-dimensional function is sought \cite{5391457}
\begin{equation*}
\label{eq:algebrac_eq}
    \vect F(x_1,x_2) = \left\{ {\begin{array}{*{20}{l}}
    a(x_1 + x_2)\\
    a(x_1 + x_2) + (x_1- x_2)((x_1 - b)^2 + x_2^2 - c)
\end{array}} \right.
\end{equation*}
where $a = 4,b = 2,c = 1$. The initial auxiliary function is set as
\begin{equation*}
    \vect G(x_1,x_2) = \left\{ {\begin{array}{*{20}{c}}
    x_1 - 2.5\\
    x_2 - 0.5
\end{array}} \right.
\end{equation*}
while the state-dependent basis function $\vect h(\kappa,\vect x)$ is
\begin{equation*}
    \vect h(\kappa,\vect x) = 
    \begin{bmatrix}
    \mathrm{e}^{x_1}\kappa^2 \\
    \mathrm{e}^{x_2}\kappa^2\\
    \end{bmatrix}
\end{equation*}
and $\Delta \kappa = 0.01 $. 

In \cite{5391457}, it is stated that if the initial condition is located inside the circle $(x_1-2)^2 + x_2^2 = 1$, like in the present case, the fixed-point homotopy function implementing PAM will fail to find the solution. This property is independently confirmed by our numerical experiment. With reference to Fig.\ \ref{fig:eg1_homo}, dashed blue line, the PAM effectively passes a singular point, after which $x_2$ goes off to infinity ($x_1$ returns to the initial point).

When the TFC-based DCM method is used (Fig.\ \ref{fig:eg1_homo}, solid red line), the limit point $\vect x_{L,0} = [1.3879, -0.9221]^\top$ is detected at $\kappa_{L,0} = 0.3738$. Here, the second-layer of the algorithm is triggered, and a new homotopy path is followed, starting from $\vect x_{0,1} = [-0.0726, -0.4492]^\top$ with 
\begin{equation*}
	\Omega_1 = 
	\begin{bmatrix}
	-9.6193 & -1.9914\\
	-3.7169 & -0.4904
	\end{bmatrix}
\end{equation*}
The new homotopy path leads smoothly to $\kappa=1$ where $\vect x^*=[0, 0]^\top$.
In this example, the TFC-based DCM method is able to detect a singular point and to successfully switch to another feasible homotopy path, which eventually converges to the solution of the objective problem.

\begin{figure}
    \centering
    \scalebox{0.6}{\includegraphics{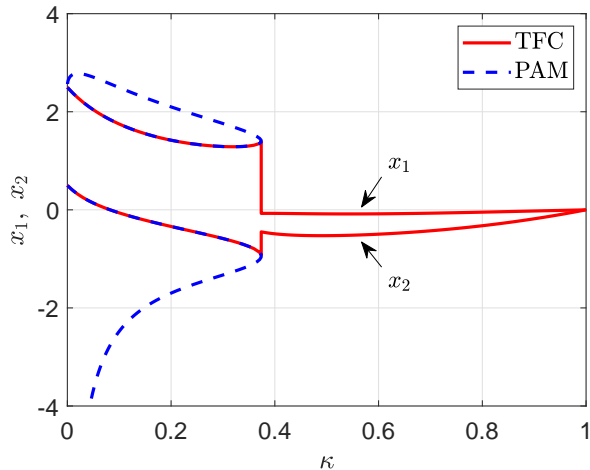}}
    \caption{Homotopy paths generated by the fixed-point method using PAM (dashed blue line) and the TFC-based DCM (solid red line) while attempting to find the zero of the function in Eq.\ \eqref{eq:algebrac_eq}.}
    \label{fig:eg1_homo}
\end{figure}

\subsection{Nonlinear Optimal Control Problem} \label{sec:example2}

Solving a nonlinear optimal control problem means find the zero of a shooting function, which solves the associated two-point boundary value problem~\cite{BrysonHo-807}. Consider the dynamical system
\begin{equation} \label{eq:ex2_dyn}
    \begin{array}{l}
    \dot x_1 = x_1 + x_2 + u_1\\
    \dot x_2 = \tan x_1^2 + u_2
    \end{array}
\end{equation}
along with the performance index
\begin{equation*}
    J = \dfrac{1}{2}\int_{0}^{t_f} \left(u_1^2+u_2^2\right) \di t
\end{equation*}
where the terminal time is $t_f = 1$, and the boundary conditions are set to $\vect x_0 = [-1,-1]^\top$ and $\vect x_f = [0,0]^\top$. An homotopy from linear to nonlinear dynamics is constructed by embedding $\kappa$ into Eq.\ \eqref{eq:ex2_dyn}, i.e.,
\begin{equation*}
    \begin{array}{l}
    \dot x_1 = x_1 + x_2 + u_1\\
    \dot x_2 = \kappa \tan x_1^2 + u_2
    \end{array}
\end{equation*}
Based on the optimal control theory\ \cite{BrysonHo-807}, the Euler--Lagrange equations are
\begin{equation}
\label{eq:NOC_1_dynamics}
    \begin{array}{l}
    \dot x_1 = x_1 + x_2 - \lambda_1\\
    \dot x_2 = \kappa \tan x_1^2 - \lambda_2\\
    \dot \lambda_1 = -\lambda_1 - 2\kappa x_1 \lambda_2/\cos^2 x_1^2 \\
    \dot \lambda_2 = -\lambda _1
    \end{array}
\end{equation}

For a given $\kappa$, the flow $\vect x(t,\vect x_0, \vect \lambda_0)$ can be obtained by integrating Eq.\ \eqref{eq:NOC_1_dynamics} with initial conditions $\vect x_0$ and $\vect \lambda_0$, where $\vect \lambda_0 = [\lambda_1(t_0),\lambda_2(t_0)]^\top$ is the initial costate vector. The zero-finding problem is to find $\vect \lambda_0$ such that $\vect F(\vect \lambda_0) = \vect 0$, where
\begin{equation*}
	\vect F(\vect \lambda_0) = \vect x(t_f,\vect x_0, \vect \lambda_0) - \vect x_f
\end{equation*}

When $\kappa = 0$, the system is linear, and the corresponding initial costate is $\vect \lambda_0 = [-2.9411, -2.0820]^\top$. In this example, the state-dependent function $\vect h(\kappa,\vect x)$ is selected as
\begin{equation*}
    \vect h(\kappa,\vect x) = 
    \begin{bmatrix}
    \mathrm{e}^{\lambda_1}\kappa^2 \\
    \mathrm{e}^{\lambda_2}\kappa^2\\
    \end{bmatrix}
\end{equation*}
and $\Delta \kappa = 0.005$. 

The simulation results are shown in Fig.~\ref{fig:NOC_example_1}, where the comparison of the homotopy paths for PAM (blue dashed line) and TFC-based DCM (red solid line) is shown in Fig.\ \ref{fig:exp2_homo_a}, whereas the optimal trajectory is shown in Fig.\ \ref{fig:eg2_tra_a}. Notice that in Fig.\ \ref{fig:exp2_homo_a} the solution curve tracked by PAM successfully passes a limit point but returns back to $\kappa \simeq 0$. PAM fails to reach the solution to the objective problem at $\kappa =1$.


When the TFC-based DCM method is used, the limit point $\vect \lambda_{L,0} = [-1.2252,-1.5880]^\top$ is detected at $\kappa_{L,0} = 0.5375$. The second layer switches to a new homotopy path starting from $\vect \lambda_{0,1} =[-1.0894, -0.7100]^\top$ with
\begin{equation*}
	\Omega_1 =
	\begin{bmatrix}
	-5.3572  & -4.5911\\
	-3.8820  &  2.3293 \\
	\end{bmatrix}
\end{equation*}
The new homotopy path leads smoothly to the solution of the objective problem, where $\vect \lambda^*(t_0)=[0.4728,-0.0739]^\top$.

\begin{figure}
     \centering
     \subfloat[]{\includegraphics[width=0.5\textwidth]{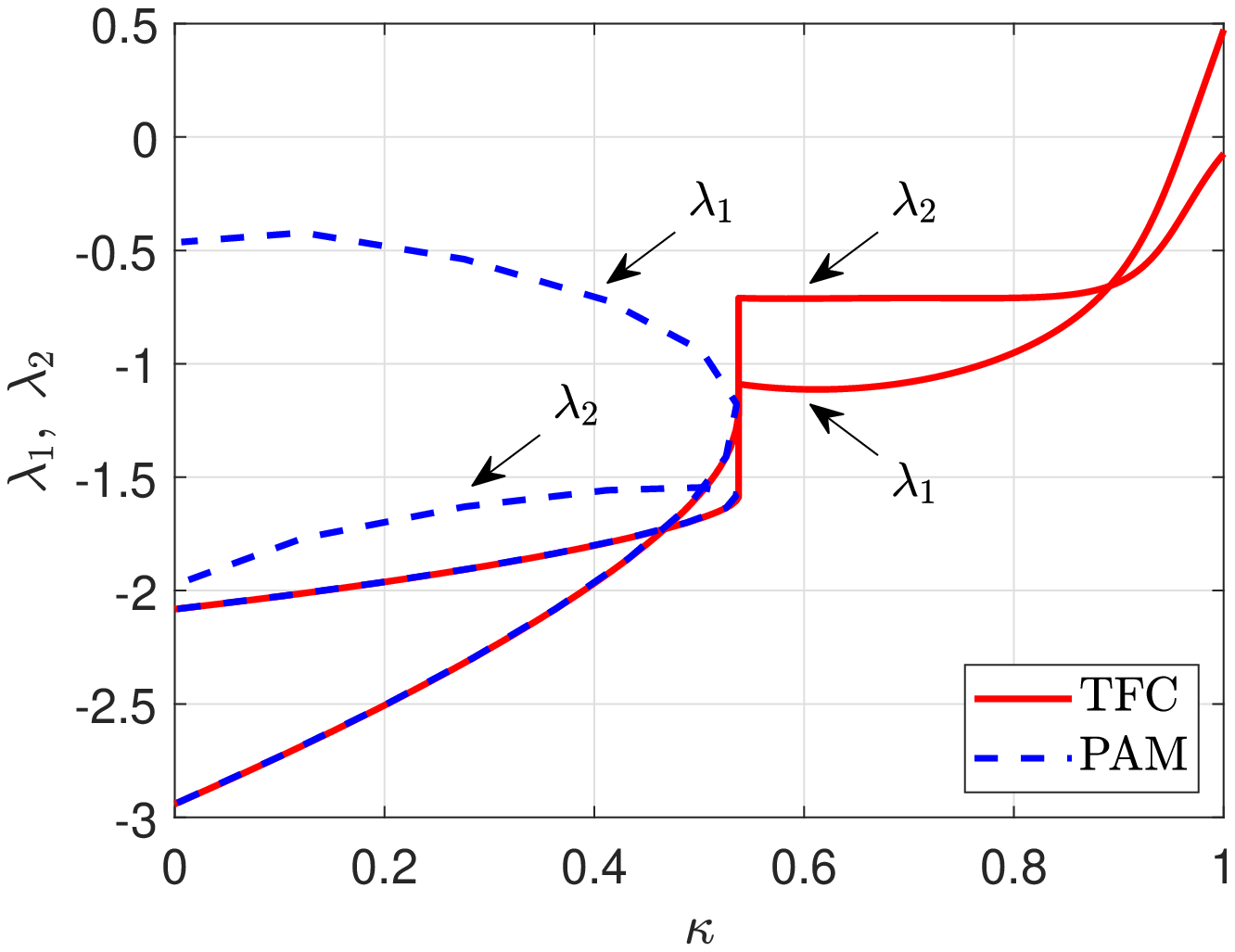}\label{fig:exp2_homo_a}}
     \hfill
     \subfloat[]{\includegraphics[width=0.5\textwidth]{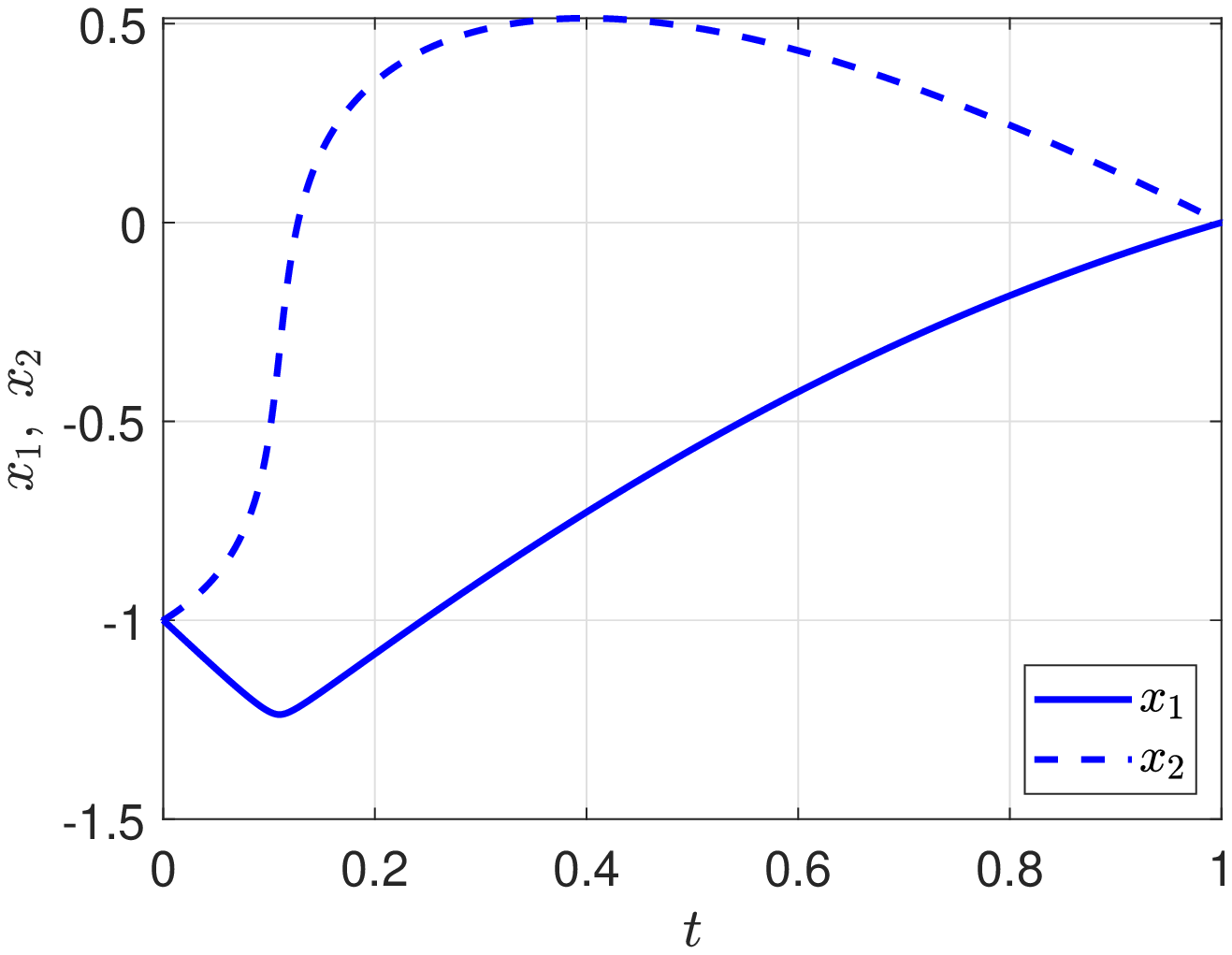}\label{fig:eg2_tra_a}}

    \caption[Simulation results]{Simulation results for the nonlinear optimal control problem. (a): Comparison of homotopy paths tracked by PAM and TFC-based DCM method; (b): Optimal trajectories $x_1(t)$ and $x_2(t)$.}
     \label{fig:NOC_example_1}
\end{figure}

\subsection{Elastic Rod Problem}

While in Sections \ref{sec:example1} and \ref{sec:example2} the issue was overcoming a singular point (Type 1, 3, and 4 in Fig.~\ref{fig:homotopy_path_type}), in this example the path goes off to infinity without encountering any limit point (Type 5 in Fig.~\ref{fig:homotopy_path_type}). The cantilever beam problem, which is to find the position $(a,b)$ of the tip of the rod given the force $Q \neq 0$ and $P =0$, has a closed-form solution in terms of elliptic integrals. The inverse problem, where the tip's position $(a,b)$ and orientation $c$ are specified, while the forces $(Q,P)$ and torque $(M)$ are to be determined, has no similar closed-form solution. It is a nonlinear problem that is difficult to solve \cite{watson1989globally}. The inverse problem is solved in this section. The dynamic equations
\begin{equation}
    \label{eqs:prob3_dyn}
    \begin{array}{l}
    \dot x =  \cos \theta\\
    \dot y =  \sin \theta\\
    \dot \theta = Q x - P y + M
    \end{array}
\end{equation}
are supported by the boundary conditions
\begin{equation*}
    x(0) = y(0) = \theta(0) = 0, \quad
    x(1) = a,\quad y(1) = b, \quad \theta(1) = c
\end{equation*}

The unknown variables are denoted as $\vect v = [Q,P,M]^\top$, and the corresponding flow is denoted as $x(t,\vect v),y(t,\vect v),\theta(t,\vect v)$. The problem is to find $\vect v^*$ such that
\begin{equation}
    \label{eq:prob3_equal_eqs}
    \vect F(\vect v^*) = 
    \begin{bmatrix}
        x(t_f,\vect v^*) - a \\
        y(t_f,\vect v^*) - b\\
        \theta(t_f,\vect v^*) - c
    \end{bmatrix}
    = \vect 0
\end{equation}

A fixed-point homotopy function is defined as
\begin{equation*}
    \label{eq:prob3_homo}
    \vect\Gamma_0(\kappa, \vect v) = (1-\kappa) \vect F(\vect v) + \kappa \vect G(\vect v) \quad \textrm{with} \quad
    \vect G(\vect v) = (\vect v - \vect v_0)
\end{equation*}
where $\vect v_0$ is the initial guess solution. The parameters are set to $a = 0$, $b = 2\pi$, $c = \pi$, and $\vect v_0= [0,0,1.85]^\top$. In this case, the solution to the objective problem in Eq.\ \eqref{eq:prob3_equal_eqs} is known to be $\vect v^* = [0,0,\pi]^\top$\ \cite{watson1981homotopy}. The Jacobian matrix of Eq.\ \eqref{eq:prob3_equal_eqs} w.r.t $\vect v$ has been computed using finite differences, and the limit threshold $T_h$ is set to $100$. The selected state-dependent basis function $\vect h(\kappa,\vect v) $ is
\begin{equation*}
    \vect h(\kappa,\vect v)  = 
    \begin{bmatrix}
    \mathrm{e}^{Q}\kappa^2\\
    \mathrm{e}^{P}\kappa^2\\
    \mathrm{e}^{M}\kappa^2
    \end{bmatrix}
\end{equation*}
and $\Delta \kappa = 0.001$.

The simulation results are shown in Fig.\ \ref{fig:NOC_example_3}, where the homotopy paths generated by PAM (blue lines) and TFC-based DCM (red lines) are shown (Fig.\ \ref{fig:pro3_homotopy_path_2} shows an enlarged view of Fig.\ \ref{fig:pro3_homotopy_path} when $\kappa\to1$). PAM is not able to reach $\vect v^*$ because the homotopy path grows indefinitely when $\kappa\to1$.

Using TFC-based DCM, the failure of the initial homotopy path is detected when $\|\vect v\|_{\infty}$ exceeds $T_h$. The point $\vect v_{{\rm I},0} = [-99.2011, -50.7766, 11.0163]^\top$ at $\kappa_{{\rm I},0} = 0.9965$ is used as initial guess for problem \eqref{eq:opt_problem}. A new start point $\vect v_{0,1} = [ -99.1967,-50.7777,11.0159]^\top$ is found, with
\begin{equation*}
\Omega_1 =
\begin{bmatrix}
0 &  0 &   -2.57 \times 10^{-5}\\
0 &  0 &   -1.59 \times 10^{-5}\\
0 &  0 &   3.2 \times 10^{-4}
\end{bmatrix}
\end{equation*}
which is very close to the initial path. Since this homotopy path excesses $T_h$ again, a second switch is attempted using $T_h/2$. The initial guess $\vect v_{{\rm I},1} =[-49.6995,-24.9227,7.8530]^\top$ at $\kappa_{{\rm I},1} = 0.9940$ is detected, and problem \eqref{eq:opt_problem} is solved gain. The new homotopy path with starting point $\vect v_{0,2} = [-51.2892,-6.0631,9.5708]^\top$ and
\begin{equation*}
\Omega_2 =
\begin{bmatrix}
0  &  0.0019  &   -0.0028\\
0 &   0.1211  &   0.0019\\
0  &  0.0280   &  0.0008 
\end{bmatrix}
\end{equation*}
is found. From this point on, the TFC-based DCM successfully reaches $\vect v^*$.



\begin{figure}
     \centering
     \subfloat[]{\includegraphics[width=0.5\textwidth]{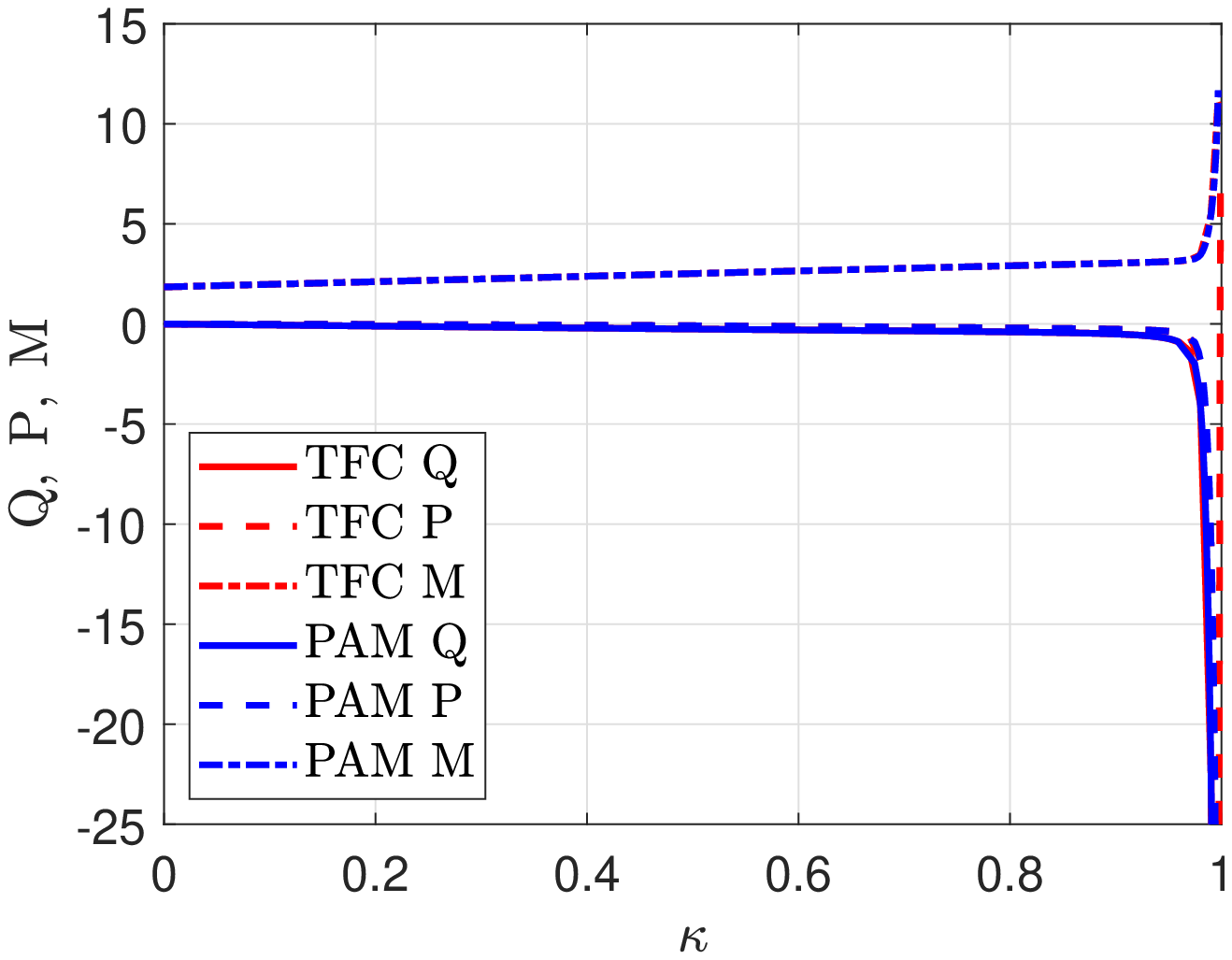}\label{fig:pro3_homotopy_path}}
     \hfill
     \subfloat[]{\includegraphics[width=0.5\textwidth]{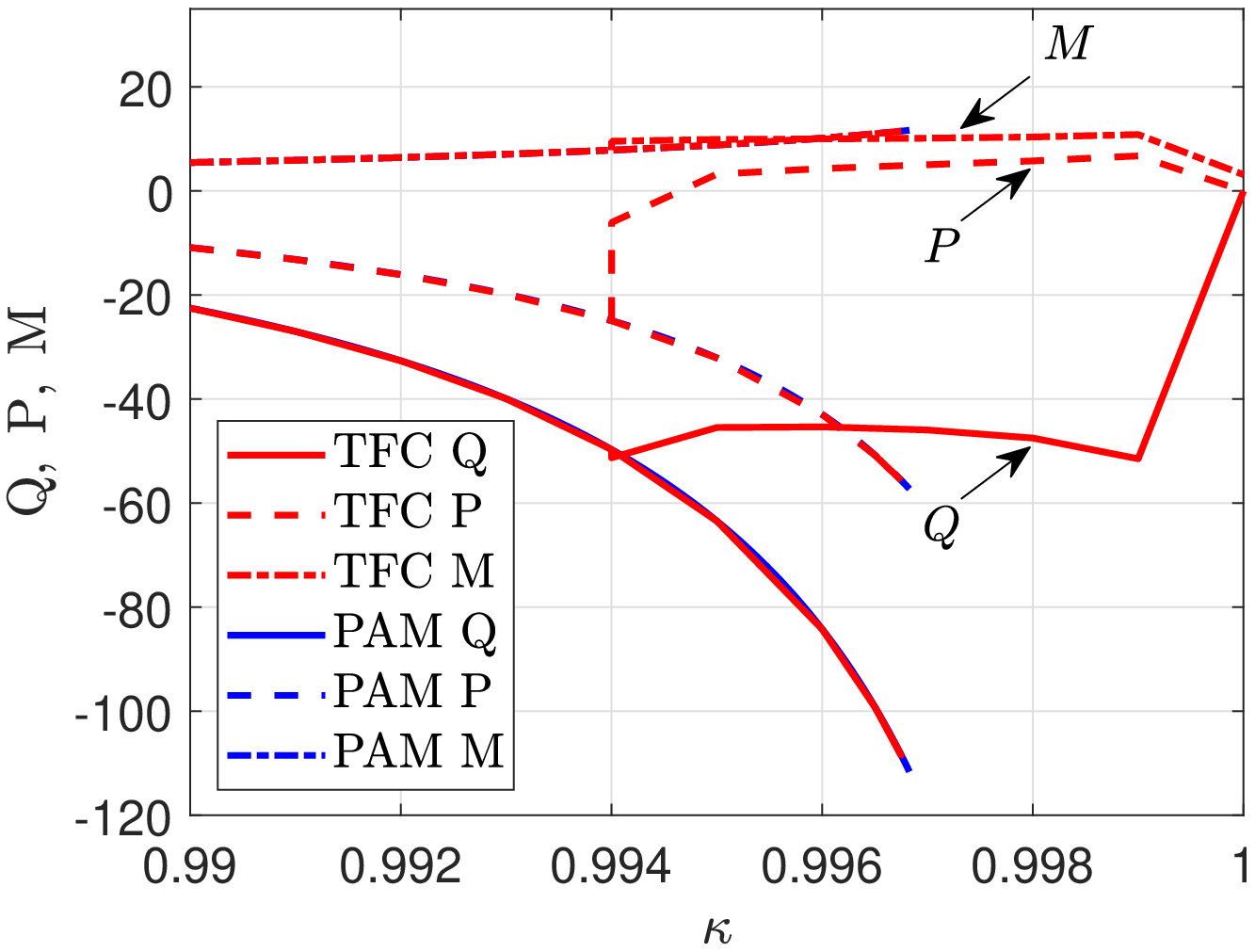}\label{fig:pro3_homotopy_path_2}}
    \caption[Simulation results]{Simulation results for elastic red problem. (a): Comparison of homotopy paths tracked by PAM (blue lines) and TFC-based DCM (red lines); (b): Zoom-in comparison of homotopy paths when $\kappa\to1$. }
     \label{fig:NOC_example_3}
\end{figure}

%
\section{Conclusion}

Homotopy is a deformation used in zero-finding problems. The idea is to connect an initial easy-to-solve problem to the final, objective problem through the solution of a number of intermediate, auxiliary problems that define the homotopy path. Traditional techniques based on pure DCM or PAM fail to reach the objective problem, e.g., when the homotopy path exhibits singular points or indefinite growth. The fate of these methods is already determined when the homotopy function is formulated and the initial condition is given.

The TFC-based homotopy function presented in this paper implicitly defines infinite homotopy paths. This property can be leveraged whenever either a singularity is found or the path tends to go off to infinity. In these cases, the algorithm is able to switch to a new homotopy path, which attempts to reach the objective problem. A two-layer TFC-based DCM algorithm has been developed to support our intuition. The effectiveness of this algorithm has been proved by solving sample problems where the traditional continuation methods fail.

The presented TFC-based homotopy function is a valuable step towards designing probability-one homotopy methods for general applications. Future work will investigate more robust strategies to find feasible homotopy paths, such as the TFC-based homotopy method from control point of view.

%
\section{Acknowledgment}
Y.W.\ acknowledges the support of the China Scholarship Council (Grant no.201706290024). The authors would like to thank Prof.\ Daniele Mortari for the fruitful discussions on the Theory of Functional Connections.
%
\section*{Appendix}
Let $\Omega_{\rm col} = \rm vec (\Omega) \in \mathbb{R}^{mn \times 1}$, where `$\rm vec$' is an operator that converts matrices into column vectors. Then,  $\Omega\vect\Gamma_{\Omega}(\kappa,\vect x)= \Tilde{\vect\Gamma}_{\Omega}(\kappa,\vect x)\Omega_{\rm col}$, where
\begin{equation*}
\Tilde{\vect\Gamma}_{\Omega}(\kappa,\vect x) \coloneqq 
\begin{bmatrix}
\Tilde{\vect h}^\top(\kappa,\vect x) &  &  & \\
& \Tilde{\vect h}^\top(\kappa,\vect x) & & \\
&  & \ddots &  \\
&  &  & \Tilde{\vect h}^\top(\kappa,\vect x)
\end{bmatrix}
\in \mathbb{R}^{n \times mn}
\end{equation*}
and 
\begin{equation*}
\Tilde{\vect h}(\kappa,\vect x) \coloneqq \left(\vect h(\kappa,\vect x) - \sum_{i=1}^{2} P_i(\kappa) Q_{i1} \vect h_{0}(\vect x) - \sum_{i=1}^{2} P_i(\kappa) Q_{i2} \vect h_{f}(\vect{x})\right)
\end{equation*}
Thus, $\Omega$ can be seen as a column vector $\vect \epsilon \in \mathbb{R}^q$ where $q = mn$. 





\section*{References}
\bibliographystyle{elsarticle-num-names}
\bibliography{MyBib.bib}







\end{document}